\documentclass[11pt]{article} 

\title{\bfseries Large nearest neighbour balls\\ in hyperbolic stochastic geometry}

\author{Moritz Otto\thanks{otto@math.au.dk, Department of Mathematics, Aarhus University, Denmark}
\and
Christoph Th\"ale\thanks{christoph.thaele@rub.de, Faculty of Mathematics, Ruhr University Bochum, Germany}}

\usepackage{amssymb}
\usepackage{amsfonts}
\usepackage{amsmath}
\usepackage{amsthm}
\usepackage{bm}
\usepackage{graphicx}
\usepackage{multicol}
\usepackage{multirow}
\usepackage{caption}
\usepackage{subcaption}
\usepackage[usenames, dvipsnames]{xcolor}
\usepackage{verbatim}
\usepackage{dsfont}
\usepackage{color}
\usepackage[numbers]{natbib}
\usepackage{relsize}
\usepackage{lmodern}
\usepackage{url}
\usepackage{MnSymbol}
\usepackage{dsfont}
\usepackage{tikz}
\usepackage{float}
\usepackage{multicol}
\usepackage{vwcol}
\usepackage{mathtools}
\usepackage[T1]{fontenc}
\usepackage[utf8]{inputenc}
\usepackage{enumitem}   

\usepackage[left=2.5cm, top=2.5cm,bottom=2.5cm,right=2.5cm]{geometry}

\theoremstyle{plain}
\newtheorem{theorem}{Theorem}
\newtheorem{lemma}[theorem]{Lemma}
\newtheorem{remark}[theorem]{Remark}
\newtheorem{proposition}[theorem]{Proposition}
\newtheorem{definition}[theorem]{Definition}
\newtheorem{corollary}[theorem]{Corollary}
\newtheorem{example}[theorem]{Example}
\newtheorem{fig}[theorem]{Figure}

\numberwithin{equation}{section}
\usepackage{caption}
\newcommand{\bthe}{\begin{theorem}}
	\newcommand{\ethe}{\end{theorem}}

\newcommand{\ben}{\begin{enumerate}}
	\newcommand{\een}{\end{enumerate}}

\newcommand{\bit}{\begin{itemize}}
	\newcommand{\eit}{\end{itemize}}

\newcommand{\beq}{\begin{equation}}
\newcommand{\eeq}{\end{equation}}

\newcommand{\ble}{\begin{lemma}}
	\newcommand{\ele}{\end{lemma}}

\newcommand{\bde}{\begin{definition}\rm}
	\newcommand{\ede}{\halmos\end{definition}}

\newcommand{\bco}{\begin{corollary}}
	\newcommand{\eco}{\end{corollary}}

\newcommand{\bpr}{\begin{proposition}}
	\newcommand{\epr}{\end{proposition}}

\newcommand{\brem}{\begin{remark}\rm}
	\newcommand{\erem}{\halmos\end{remark}}

\newcommand{\bproof}{\begin{proof}[Proof]}
	\newcommand{\eproof}{\end{proof}}

\newcommand{\bexam}{\begin{example}\rm}
	\newcommand{\eexam}{\halmos\end{example}}

\newcommand{\bexamwh}{\begin{example}\rm}
	\newcommand{\eexamwh}{\end{example}}

\newcommand{\bfi}{\begin{fig}}
	\newcommand{\efi}{\end{fig}}

\newcommand{\btab}{\begin{tab}}
	\newcommand{\etab}{\end{tab}}

\newcommand{\beao}{\begin{eqnarray*}}
	\newcommand{\eeao}{\end{eqnarray*}\noindent}

\newcommand{\beam}{\begin{eqnarray}}
\newcommand{\eeam}{\end{eqnarray}\noindent}

\newcommand{\ovr}{\begin{array}}
	\newcommand{\barr}{\begin{array}}
		\newcommand{\earr}{\end{array}}
	
	\newcommand{\bdis}{\begin{displaymath}}
	\newcommand{\edis}{\end{displaymath}\noindent}

	\DeclareMathOperator{\arcosh}{arcosh}

	\def\N{{\mathbb N}}
	
	\def\P{{\mathbb P}}
	\def\E{{\mathbb E}}
	\def\R{{\mathbb R}}

	\def\H{{\mathbb H}}
	
	\def\phi{\varphi}
	
	\def\cals_+{{\cals_+}}
	
	\def\calh{{\mathcal{H}}}

	\def\cals{{\mathcal{S}}}

	\def\1{\mathds{1}}

		\newcommand{\XX}{\mathbb{X}}

	\newcommand{\halmos}{\quad\hfill\mbox{$\Box$}}

\definecolor{plum}{cmyk}{0.50,1,0,0}
\definecolor{TealBlue}{cmyk}{0.86,0,0.34,0.02}
\definecolor{OliveGreen}{cmyk}{0.64,0,0.95,0.40}

\makeatletter
\let\@fnsymbol\@alph
\makeatother

\begin{document}

\date{}

\maketitle

\begin{abstract}
	\noindent  Consider a stationary Poisson process in a $d$-dimensional hyperbolic space. For $R>0$ define the point process $\xi_R^{(k)}$ of exceedance heights over a suitable threshold of the hyperbolic volumes of $k$th nearest neighbour balls centred around the points of the Poisson process within a hyperbolic ball of radius $R$ centred at a fixed point. The point process $\xi_R^{(k)}$ is compared to an inhomogeneous Poisson process on the real line with intensity function $e^{-u}$ and point process convergence in the Kantorovich-Rubinstein distance is shown. From this, a quantitative limit theorem for the hyperbolic maximum $k$th nearest neighbour ball with a limiting Gumbel distribution is derived.
	\bigskip
	\\
	{\bf Keywords}. {Geometric extreme value theory, hyperbolic stochastic geometry, nearest neighbour balls, Poisson process approximation.}\\
	{\bf MSC}. 52A55, 60D05, 60G55.
\end{abstract}


\section{Introduction}

The study of extreme values, or more generally processes of exceedance heights and associated order statistics, is a classical topic in probability theory. A systematic study of extreme values for random geometric systems is more recent and we refer, for example, to \citep{BonnetChenavier,CalkaChenavier,ChenavierHemsley,Chenavier} for particular results on the Poisson-Voronoi, -Delaunay or -line tessellation, to \cite{Jammalamadaka,Schremp} for distinguished results on random interpoint distances, and to \cite{BSY21,DST16,PianoforteSchulte,PianoforteSchulte2,SchulteThaele12,SchulteThaele16} for general approaches leading to various other stochastic-geometric applications. A systematic study of random processes of exceedance heights in stochastic geometry is the content of \cite{BSY21,DST16,Otto}. 

In this paper we are interested in quantitative limit theorems for so-called large $k$th nearest neighbour balls, another classical stochastic geometry model whose investigation goes back to \cite{H82} and which has recently been studied in \cite{BSY21,CHO21,GHW19}. In particular, nearest neighbour balls (that is, $k$th nearest neighbour balls with $k=1$) can be regarded as spatial analogues of the concept of spacings in dimension one. In its simplest form the model can be described as follows: Take a sequence $(X_i)_{i\in\N}$ of independent random points which are uniformly distributed on the $d$-dimensional unit cube $[0,1]^d\subset\R^d$. For $n\geq 1$, $k\in\{1,\ldots,n\}$ and $i\in\{1,\ldots,n\}$ let {$r(i,n)$} be the distance of $X_i$ to its $k$th nearest neighbour among the points $X_1,\ldots,X_n$, where the distance is { understood in the Euclidean sense}. Then for $t\in\R$ define the random variable
$$
C_n := \sum_{i=1}^n\mathds{1}\Big\{\calh_e^d(B_e(X_i,r(i,n)))\geq {t+\log n\over n}\Big\},
$$
where $\calh_e^d$ stands for the $d$-dimensional Hausdorff measure and $B_e(X_i,r(i,n))$ for the $d$-dimensional ball of radius $r(i,n)$ centred at $X_i$ with respect to the Euclidean structure on $\R^d$. In other words, $C_n$ counts the number of exceedances of volumes of {$k$th} nearest neighbour balls that are larger than the threshold $(t+\log n)/n$. It follows from the results in \cite{BSY21,CHO21,GHW19} that, after suitable normalization, $C_n$ converges in distribution, as $n\to\infty$, to a Poisson random variable with mean $e^{-t}$. Moreover, the rate of convergence, measured in the total variation distance, is of order $(\log\log n)/\log n$. This result immediately leads to a limit theorem for the maximum volume $M_n:=\max\{\calh_e^d(B_e(X_i,r(i,n))):1\leq i\leq n\}$ of the $k$th nearest neighbour balls, which says that $nM_n$ converges, after suitable centring, to a Gumbel distribution, as $n\to\infty$. A similar result holds if the sample of $n$ independent random points is replaced by a homogeneous Poisson process in $[0,1]^d$ with intensity $n$.

While the results and references just mentioned deal with $k$th nearest neighbour balls in a $d$-dimensional \textit{Euclidean} space, we follow another line of current research in stochastic geometry and introduce and study a similar model in a $d$-dimensional \textit{hyperbolic} space $\H^d$ of constant negative curvature $-1$. Random geometric systems in such a non-Euclidean set-up have so far been studied in the context of random polytopes \cite{BesauRosenThaele,BesauThaele,GodlandKabluchkoThaele}, random graphs \cite{BodeEtAl,FountoulakisMuller,FountoulakisVanDenHornEtAl,OwadaYogesh} and tessellations \cite{GodlandKabluchkoThaele,Isokawa}. However, the study of extreme values in hyperbolic stochastic geometry has so far left no trace in the existing literature. The present paper can be understood as a first attempt in this direction. Moreover, since the results for $k$th nearest neighbour balls in Euclidean space have found applications in goodness-of-fit testing for point processes \cite{H82,H83}, our contributions can also be of interest for similar studies in hyperbolic space.

In principle it is possible to rephrase the Euclidean model of large $k$th nearest neighbour balls in a hyperbolic space, where the cube (which does not exist in hyperbolic geometry) is replaced by a hyperbolic ball of radius one, say. However, in this case, one can localize the problem and work with approximations in the corresponding tangent spaces. Within these tangent spaces the model is Euclidean again and we get back a result similar to that in \cite{BSY21,CHO21,GHW19}. For this reason, we modify the set-up as follows: We start with a stationary Poisson process in $\H^d$ and look at large $k$th nearest neighbour balls associated with points in a family of hyperbolic balls of radius $R\to\infty$. Up to a rescaling, in a Euclidean space this set-up is the same as fixing the radius of the ball and increasing the intensity of the Poisson process (or equivalently the number of points). However, this is no more the case in a hyperbolic space. Even more, since the problem in this form cannot locally be approximated by Euclidean models in tangent spaces, we will arrive at results which are of a different nature and `feel' the negative curvature of the underlying space. 

In the next section we formally describe the framework we work with and present our results.

\section{Set-up and results}

Fix a dimension parameter $d\in\N$ and consider a $d$-dimensional hyperbolic space $\H^d$ together with the intrinsic (Riemannian) metric $d_h$ and the corresponding $d$-dimensional Hausdorff measure $\calh^d$. Although all our results are independent of a concrete model for $\H^d$ (as can be seen from the fact that non of our arguments or computations rely on a specific model), for concreteness one may consider the Beltrami-Klein model in which $\H^d$ is identified with the open Euclidean unit ball $\mathbb{B}^d$ and the Riemannian metric is given by
$$
{\rm d}s^2 = {{\rm d}x_1^2+\ldots+{\rm d} x_d^2\over 1-x_1^2-\ldots-x_d^2}+{(x_1{\rm d} x_1+\ldots +x_d{\rm d} x_d)^2\over(1-x_1^2-\ldots-x_d^2)^2},
$$
see \cite{CannonEtAlHyperbolic,R2019} for details and further models for $\H^d$. While in this model hyperbolic hyperplanes are non-empty intersections of Euclidean hyperplanes with $\mathbb{B}^d$, hyperbolic balls are represented by Euclidean ellipsoids, see Figure \ref{fig:bild}. For $z \in \mathbb{H}^d$ and $r>0$ let $B(z,r):=\{x \in \mathbb{H}^d:\,d_h(x,z)\le r\}$ denote the closed hyperbolic ball of radius $r$ centred at $z$. We abbreviate $B_r:=B(p,r)$, where $p \in \mathbb{H}^d$ is some arbitrary fixed point, referred to as the origin of $\H^d$. 

\begin{figure}[t]
	\centering
	\vspace{-2cm}
	\includegraphics[width=0.8\columnwidth]{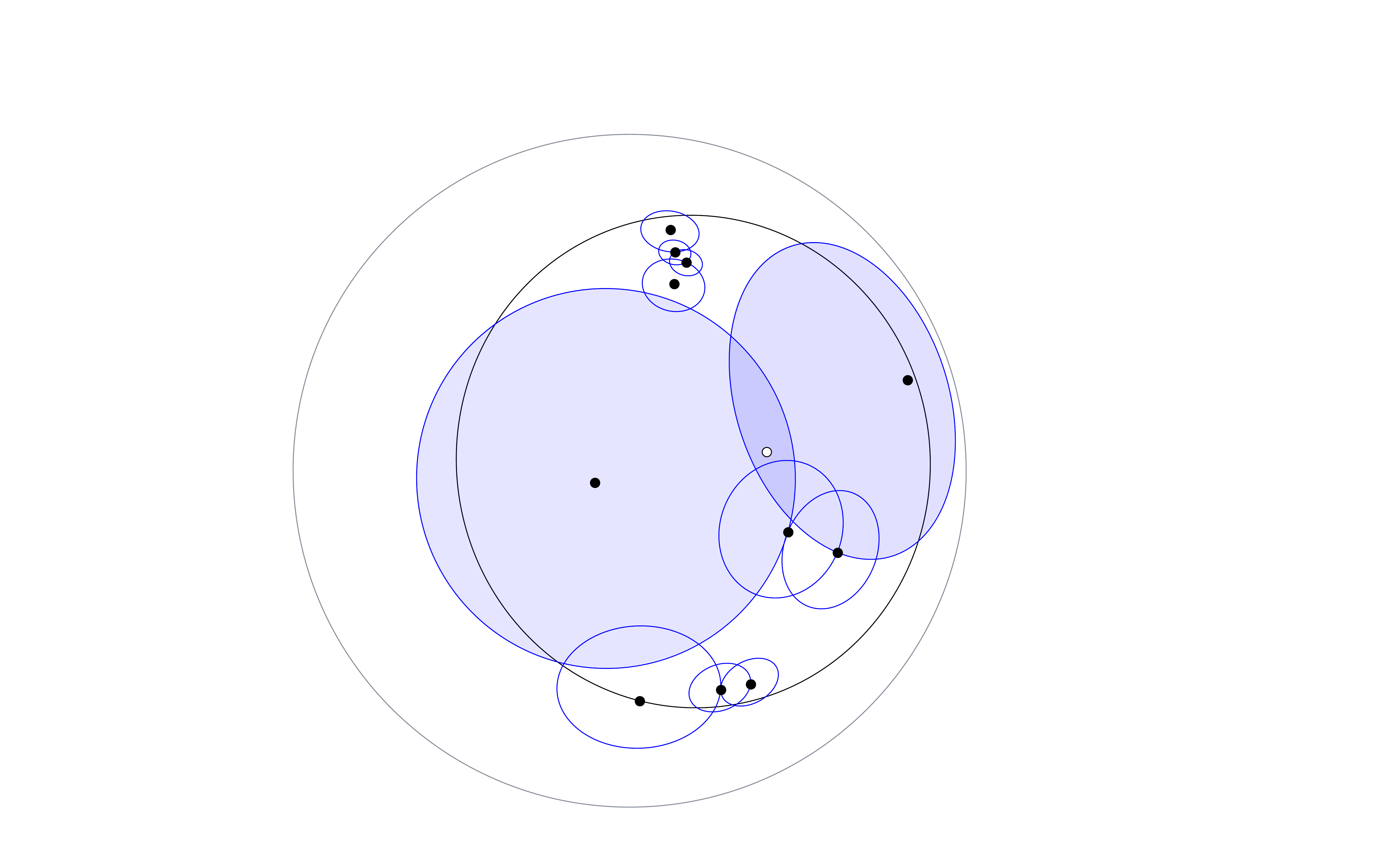}
	\caption{Construction of nearest neighbour balls ($k=1$) in the Beltrami-Klein model for the hyperbolic plane. The ball $B_R$ is shown in black with a white centre, the black points are points from $\eta$ together with their nearest neighbour balls. The area of the two blue balls exceed the value $v_1(R)$.}
	\label{fig:bild}
\end{figure}

Let $\eta$ be a Poisson process in $\mathbb{H}^d$, $d\geq 2$, with intensity measure $\mathcal{H}^d$. We note that $\eta$ is stationary in the sense that its distribution is invariant under all isometries of the hyperbolic space. For $R >0$ and $k\in\N$ let 
\begin{align}\label{eq:vkr}
	v_k(R) := R\,(d-1)+{(k-1)}\log (R\,(d-1))-\log \left(\frac{{(k-1)}!2^{d-1}(d-1)}{\omega_d}\right),
\end{align}
where $\omega_d={2\pi^{d/2}\over\Gamma({d\over 2})}$ is the surface area of the $(d-1)$-dimensional Euclidean unit sphere. For $x\in \mathbb{H}^d$, $k\in\N$ and a general, locally finite and simple counting measure $\mu$ on $\mathbb{H}^d$ let $r(x,k,\mu)$ denote the hyperbolic distance to the $k$th nearest neighbour of $x$ in $\mu$. In the focus of our results is the point process
\begin{align}\label{eq:PPExceedances}
	\xi_{R}^{(k)}:=\sum_{x \in \eta \cap B_R} \delta_{\mathcal{H}^d(B_{r(x,k,\eta-\delta_x)})-v_k(R)},\qquad R>0,
\end{align}
on the real line $\R$. This process describes the exceedance heights over the threshold $v_k(R)$ of the volumes of $k$th nearest neighbour balls centred around the points of $\eta$ within a ball $B_R$ of radius $R$, see Figure \ref{fig:bild}. {We give a brief heuristic argument which explains at least for $k=1$ that $v_k(R)$ given in \eqref{eq:vkr} is the correct threshold to expect Poisson approximation for $\xi_R^{(k)}$. Given some point of the Poisson process $\eta$, the probability that the hyperbolic ball of volume $v_1(R)+c$ around this point does not contain any further points of $\eta$ is $\exp(-v_1(R)+c)$. Assuming that all such balls around all points of $\eta \cap B_R$ behave asymptotically independently (which is, of course, not true and requires justification), the probability $\mathbb{P}[\mathcal{E}(R)]$ of the event $\mathcal{E}(R)$ that in $\eta \cap B_R$ there is no 1-nearest neighbour ball of radius larger than $v_1(R)+c$ should be approximately equal to $\exp(-\mathcal{H}^d(B_R) \exp(-v_1(R)-c))$, where $\mathcal{H}^d(B_R)$ is for large $R$ approximately the number of points of $\eta \cap B_R$. Using that $\mathcal{H}^d(B_R) \exp(-v_1(R))\to 1$ as $R \to \infty$ (see \eqref{volbou}) we obtain from the classical Poisson limit theorem that $\mathbb{P}[\mathcal{E}(R)]$ converges to the Gumbel limit $\exp(-e^{-c})$ as $R\to \infty$.}

Our main results quantify on the positive real half-axis $\mathbb{R}_+$ the approximation of $\xi_{R}^{(k)}$ by a suitable Poisson process, where we distinguish the cases $k=1$ and $k\geq 2$. The distance is thereby measured by means of the so-called Kantorovich-Rubinstein distance which for two finite simple counting measures $\mu_1$ and $\mu_2$ on $\H^d$ is given by
$$
\mathbf{d_{KR}}(\mu_1,\mu_2) := \sup_{h}\big(\E[h(\mu_1)]-\E[h(\mu_2)]\big),
$$
where the supremum is taken over all measurable $1$-Lipschitz functions with respect to the total variation distance on the space of finite simple counting measures on $\H^d$.

\begin{theorem}\label{Th1}
	Let $\zeta$ be an inhomogeneous Poisson process on ${\mathbb{R}}$ with intensity measure $\mathbb{E}\zeta$ given by $\mathbb{E} \zeta ((u,\infty))=e^{-u},\, {u\in \mathbb{R}}$.  Let $ {c\in \mathbb{R}}$.
	
	\begin{itemize}
		\item[(i)] Suppose that $k=1$. Then there are constants $C_{1,d},R_{1,d}>0$ only depending on $d$ and $c$ such that for all $R\ge R_{1,d}$,
		\begin{align*}
			\mathbf{d_{KR}}(\xi_R^{(1)} \cap (c,\infty),\zeta \cap (c,\infty)) \le \begin{cases}
			C_{1,d} \,Re^{-R(d-1)/2} &: d\leq 5\\
			C_{1,d} \,e^{-2R} &: d\geq 6.
			\end{cases}
		\end{align*}
		
		\item[(ii)] Suppose that $k\geq 2$. Then there are constants $C_{k,d},R_{k,d}>0$ only depending on $d$, $k$ and on $c$ such that for all $R\geq R_{k,d}$,
		\begin{align*}
			\mathbf{d_{KR}}(\xi_R^{(k)} \cap (c,\infty),\zeta \cap (c,\infty)) \le {C_{k,d} {\log R \over R}}.
		\end{align*}
	\end{itemize}
	In particular, for any $k\in\N$ the point process $\xi_R^{(k)}$ converges in distribution to the Poisson process $\zeta$, as $R\to\infty$.
\end{theorem}

\begin{remark}\rm 
	\begin{itemize}
	\item[(i)] 	{As a generalization of Theorem 1, one can prove that the marked point processes
		\begin{align*}
		\sum_{x \in \eta \cap B_R} \delta_{(x,\mathcal{H}^d(B_{r(x,k,\eta-\delta_x)})-v_k(R))},\qquad R>0,
		\end{align*}
restricted to some interval $(c,\infty)$ converge to a Poisson process on the product space $\mathbb H^d \times \mathbb{R}$, as $R\to \infty$.}  
		
		\item[(ii)] {The distinction between $k=1$ and $k\geq 2$ and the qualitatively different results in these cases reflect, in a sense, the growth of $v_k(R)-R(d-1)$. While $v_1(R)-R(d-1)$ is constant in $R>0$, we find that $v_k(R)-R(d-1)$ grows logarithmically in $R$ for $k\ge 2$.}
				
		\item[(iii)] We leave it as an open problem to decide whether (or not) the bounds in Theorem \ref{Th1} are optimal. However, we remark at this point that the bound in Theorem  \ref{Th1} for $k\geq 2$ are in accordance with the bounds of the analogous problem in a Euclidean space (see \cite[Theorem 6.4]{BSY21} and \cite[Theorem 1.2]{CHO21}). Indeed, note that by Lemma \ref{lem:LowerBoundVolumeBall} the statement of Theorem \ref{Th1}(ii) is equivalent to
			\begin{align*}
		\mathbf{d_{KR}}(\xi_R^{(k)} \cap (c,\infty),\zeta \cap (c,\infty)) &\le \tilde C_{k,d} \,{\log \log \mathcal H^d(B_R) \over \log \mathcal H^d(B_R)},\qquad R >R_{k,d},
	\end{align*} 
	with some positive constants $\tilde C_{k,d}$ depending on $c$, $k$ and $d$.

	\item[(iv)] If we replace $v_1(R)$ in \eqref{eq:vkr} by $\tilde{v}_1(R):=\log\mathcal{H}^d(B_R)$, then the bound in Theorem \ref{Th1}(i) can be improved to
	$$
				\mathbf{d_{KR}}(\xi_R^{(1)} \cap (c,\infty),\zeta \cap (c,\infty)) \le \tilde{C}_{1,d}\,Re^{-R(d-1)/2},\qquad R>\tilde{R}_{1,d},
	$$
	or equivalently,
	\begin{align*}
		\mathbf{d_{KR}}(\xi_R^{(1)} \cap (c,\infty),\zeta \cap (c,\infty)) \le \hat C_{1,d} \,{ \log \mathcal H^d(B_R) \over  \mathcal H^d(B_R)^{1/2}},\qquad R > \tilde R_{1,d},
	\end{align*}
	for \textit{all} $d\geq 2$, where $\tilde{C}_{1,d},\hat C_{1,d},\tilde R_{1,d}\in(0,\infty)$ are constants only depending on $c$ and $d$. In this form, the bound is in accordance with \cite{BSY21} if in the proof of Theorem 6.4 therein one systematically exploits that $k=1$\footnote{In fact, in \cite[Equation (6.12)]{BSY21} the integral term vanishes and the two remaining terms decay exponentially in $n$. The bounds $E_1\leq c_1n^{-1}$ and $E_2\leq c_2n^{-1}\log n$ remain unchanged, while $E_{3,1}=0$ and $E_{3,2}\leq c_3n^{-1/2}\log n$ in the terminology of \cite{BSY21}. This eventually yields a bound of order $n^{-1/2}\log n$ for the Kantorovich-Rubinstein distance.}. We note that the numerically much more tractable expression $v_1(R)$ in \eqref{eq:vkr} can be regarded as a first-order approximation of $\tilde{v}_1(R)$. It is the error in this approximation, which makes the case distinction in Theorem \ref{Th1}(i) unavoidable. We decided to work with $v_k(R)$ as given by \eqref{eq:vkr} in order to ease comparison with the Euclidean case.
	\end{itemize}
\end{remark}

The following extreme value statement for the distribution of the maximum volume of $k$th nearest neighbour balls in $\eta \cap B_R$ is a direct consequence of Theorem \ref{Th1}. It can be understood as the hyperbolic analogue to results for the asymptotic distribution of maximum $k$th nearest neighbour balls in Euclidean space, see \cite{BSY21}, Section 6.2, and \cite{CHO21} for general $k \in \mathbb{N}$ as well as \cite{GHW19} for an elementary proof in the special case $k=1$. 

\begin{corollary}
	Let $k\in\N$, {$c \in \mathbb{R}$} and denote by $C_{k,d},R_{k,d}>0$ the constants appearing in Theorem \ref{Th1}.
	\begin{itemize}
		\item[(i)] Suppose that $k=1$. Then for all $R\ge R_1$ we have that 
		\begin{align*}
			\Bigl|\P\Big(\max_{x \in \eta \cap B_R} \mathcal{H}^d(B_{r(x,1,\eta-\delta_x)})-v_1(R) \le c\Big)-\exp({-}e^{-c})\Bigl| \le \begin{cases}
				C_{1,d} \,Re^{-R(d-1)/2} &: d\leq 5\\
				C_{1,d} \,e^{-2R} &: d\geq 6.
			\end{cases}
		\end{align*}
		\item[(ii)] Suppose that $k\geq 2$. Then for all $R\geq R_k$ we have that 
		\begin{align*}
			\Bigl|\P\Big(\max_{x \in \eta \cap B_R} \mathcal{H}^d(B_{r(x,k,\eta-\delta_x)})-v_k(R) \le c\Big)-\exp({-}e^{-c})\Bigl| \le {C_{k,d} {\log R \over R}}.
		\end{align*}
	\end{itemize}	
	In particular, for any $k\in\N$ the random variable $\max_{x \in \eta \cap B_R} \mathcal{H}^d(B_{r(x,k,\eta-\delta_x)})-v_k(R)$ converges in distribution to a Gumbel distribution, as $R\to\infty$.
\end{corollary}
\begin{proof}
	In view of the inequality
	\begin{align*}
		|\P(\xi_R^{(k)} \cap (c,\infty)=\emptyset)-\P(\zeta\cap (c,\infty)=\emptyset)|\le\mathbf{d_{KR}}(\xi_R^{(k)}\cap (c,\infty),\zeta \cap (c,\infty) ),
	\end{align*}
	the claim follows directly from Theorem \ref{Th1}.
\end{proof}

The remaining parts of this paper are structured as follows. In Section \ref{sec:prelim} we provide some necessary background material from hyperbolic geometry and prove some auxiliary geometric estimates. We also rephrase there a general bound for quantitative Poisson approximation from \cite{BSY21} on which the proof of Theorem \ref{Th1} is based. The latter is provided in Section \ref{sec:Proof2}.

\section{Preliminaries}\label{sec:prelim}

\subsection{Hyperbolic geometry}

In this section we collect some preliminary materials from hyperbolic geometry, which are relevant in our context, and refer to the monograph \cite{R2019} and the survey article \cite{CannonEtAlHyperbolic} for further information. We recall that $B(z,r)$ stands for a $d$-dimensional geodesic ball of radius $r>0$ centred at $z\in\H^d$. If $z=p$ we simply write $B_r$ for $B(p,r)$. The volume of $B_r$ is given by
\begin{align}\label{eq:VolBall}
	\calh^d(B_r) = \omega_d\int_0^r\sinh^{d-1}(u)\,\mathrm{d}u,
\end{align}
where we recall that $\omega_d$ is the surface area of the $(d-1)$-dimensional Euclidean unit sphere, see \cite[Equation (3.26)]{R2019}. Identity \eqref{eq:VolBall} is a consequence of the polar integration formula in hyperbolic geometry \cite[pp. 123-125]{Chavel}, which says that
\begin{align}\label{eq:PolarIntegration}
	\int_{\H^d} f(x)\,\calh^d({\rm d}x) = \omega_d\int_{\mathbb{S}_p^{d-1}}\int_0^\infty \sinh^{d-1}(u)\,f(\exp_p(uv))\,{\rm d}u\sigma_p({\rm d}v),
\end{align}
where $\mathbb{S}_p^{d-1}$ is the $(d-1)$-dimensional unit sphere in the tangent space $T_p$ at $p$, $\sigma_p$ the normalized spherical Lebesgue measure on $\mathbb{S}_p^{d-1}\subset T_p$ and $\exp_p(uv)$ stands for the point in $\H^d$ arising by applying the exponential map $\exp_p:T_p\to\H^d$ to the point $uv\in T_p$. In particular, we have the following bounds for $\calh^d(B_r)$. 

\begin{lemma}\label{lem:LowerBoundVolumeBall}
	Let $d\geq 2$. Then there are constants $\Gamma_d,\gamma_d>0$ only depending on $d$ such that $\gamma_d e^{r(d-1)}\leq \calh^d(B_r)\leq\Gamma_d e^{r(d-1)}$ for all $r\geq 2$.
\end{lemma}
\begin{proof}
	It is elementary to check that $\sinh(u)\geq u$ for any $u\geq 0$ and that $\sinh(u)\geq e^u/3$ for $u\geq 1$. Since $d\geq 2$ we  conclude that
	\begin{align*}
		\calh^d(B_r) &\geq \omega_d\int_0^1 u^{d-1}\,\mathrm{d}u + {\omega_d\over 3^{d-1}}\int_1^r e^{(d-1)u}\,\mathrm{d}u \\
		&= {\omega_d\over d} + {\omega_d\over (d-1)3^{d-1}}(e^{r(d-1)}-e^{d-1}) \\
		&\geq {\omega_d\over (d-1)3^{d-1}}(e^{r(d-1)}-e^{d-1}) \\
		& \geq \gamma_d e^{r(d-1)}
	\end{align*}
	with the choice $\gamma_d={\omega_d\over 2(d-1)3^{d-1}}$, where for the last inequality we used that $r\geq 2$. On the other hand, $\sinh(u)\leq e^{u}/2$ for all $u\geq 0$ and we obtain
	\begin{align*}
		\calh^d(B_r) \leq {\omega_d\over 2^{d-1}}\int_0^r e^{(d-1)u}\,\mathrm{d}u \leq \Gamma_d e^{r(d-1)}
	\end{align*}
	with $\Gamma_d={\omega_d\over (d-1)2^{d-1}}$.
\end{proof}

The following lemma is an essential ingredient of the proof of Theorem \ref{Th1}. It provides a bound for the volume of the difference of two nearby hyperbolic balls with the same radius.
\begin{figure}[t]
	\centering
	\vspace{-1cm}
	\includegraphics[width=0.8\columnwidth]{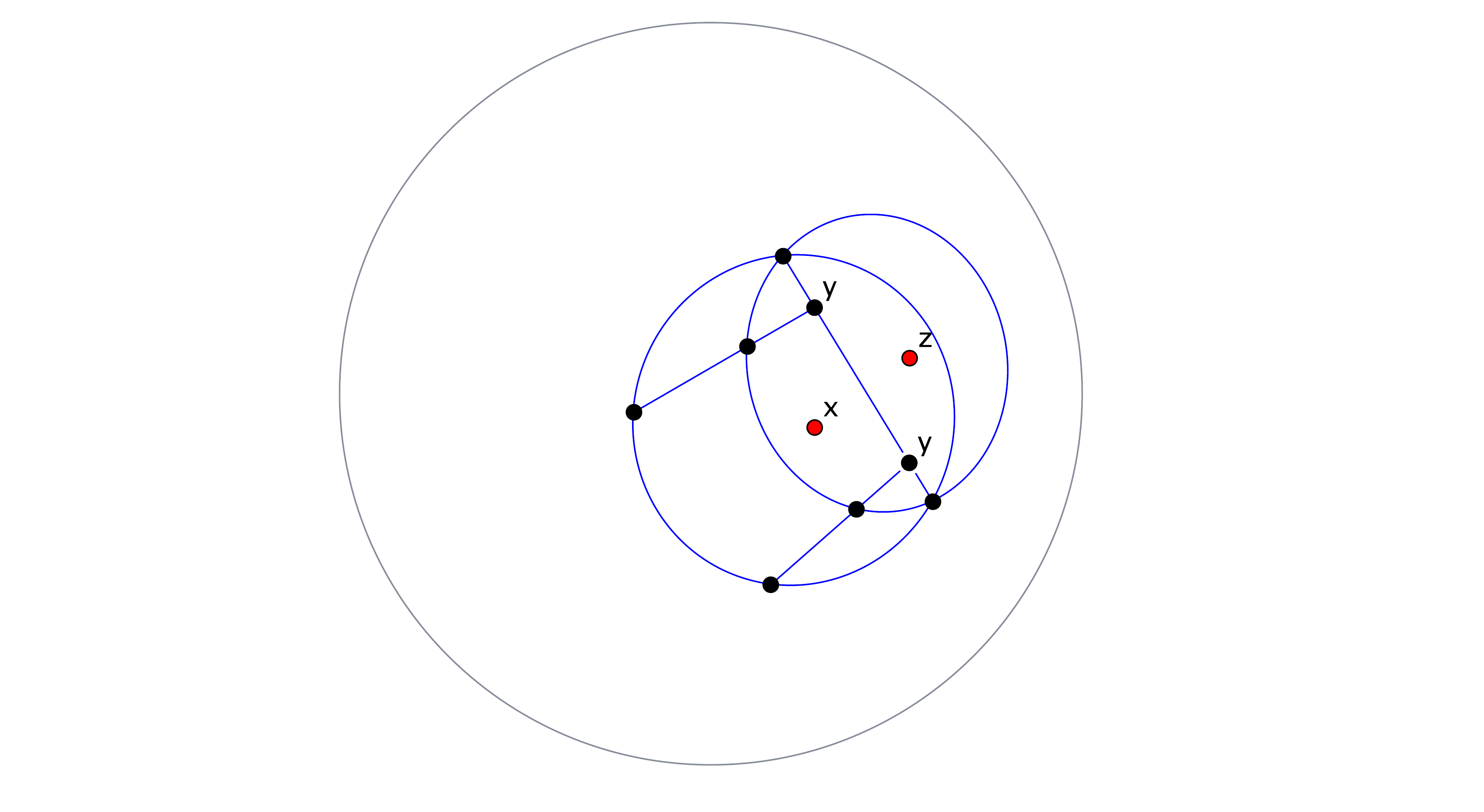}
	\caption{Illustration in the Beltrami-Klein model for the hyperbolic plane of the argument used in the proof of Lemma \ref{voldif}. Shown are the two balls $B(x,r)$ and $B(z,r)$ with $d_h(x,z)<r$, as well as two perpendiculars through points of $B_r(z,x)$.}
	\label{fig:lemma}
\end{figure}

\begin{lemma} \label{voldif}
	Let $x,z \in \mathbb{H}^d$ and $0<s:=d_h(x,z)\leq r$ with $r-s/2\geq 2$. 
	\begin{itemize}
		\item[(i)] It holds that
		\begin{align*}
		\alpha_1se^{(d-{1})(r-s/2)} \le \mathcal{H}^d(B(z,r)\setminus B(x,r))  \le \alpha_2 se^{(d-{1})r},
		\end{align*}
	where $\alpha_1,\alpha_2\in(0,\infty)$ are constants only depending on $d$.
		\item[(ii)] For $s>0$ we have
		\begin{align}
			\omega_d\left[\frac{e^{s(d-1)}}{(d-1)2^{d-1}}-\frac{(d-1)e^{s(d-3)}}{(d-3)2^{d-1}}\right]\le \mathcal{H}^d(B_s)\le \omega_d \frac{e^{s(d-1)}}{(d-1)2^{d-1}}. \label{volbou}
		\end{align}
	\end{itemize}
\end{lemma} 
\begin{proof}
	We start by observing that the boundary of the intersection $B(z,r)\cap B(x,r)$ is a $(d-2)$-dimensional sphere of radius $\overline{r}:=\arcosh \left(\frac{\cosh(r)}{\cosh(s/2)}\right)$ according to \cite[Theorem 3.5.3]{R2019}. Let $B_r(z,x)$ be the corresponding $(d-1)$-dimensional ball.
	For each $y\in B_r(z,x)$ let $L(y)$ be the hyperbolic line through $y$ which is orthogonal to the hyperbolic hyperplane containing $B_r(z,x)$. By construction, the set $(B(z,r)\setminus B(x,r))\cap L(y)$ is a hyperbolic segment of length $r-(r-s)=s$ and it follows that
	\begin{align*}
		\mathcal{H}^d(B(z,r) \setminus B(x,r)) &\geq \int_{B_r(z,x)}\calh^1((B(z,r)\setminus B(x,r))\cap L(y))\,\cosh(d_h(L(y),m))\,\calh^{d-1}(\mathrm{d}y)\\
		&= s\int_{B_r(z,x)}\cosh(d_h(L(y),m))\,\calh^{d-1}(\mathrm{d}y),
	\end{align*}
	see Figure \ref{fig:lemma}, where $m$ stands for the centre of $B_r(z,x)$ and $d_h(L(y),m)$ for the hyperbolic distance from $L(y)$ to $m$.
	
	To derive a lower bound for the integral we use that $\overline{r} \ge r-s/2$ according to \cite[Lemma 6]{HHT2021} and our assumption. Using the polar integration formula \eqref{eq:PolarIntegration} within the hyperbolic hyperplane containing $B_r(x,z)$ this leads to
	\begin{align*}
	&\int_{B_r(z,x)}\cosh(d_h(L(y),m))\,\calh^{d-1}(\mathrm{d}y)= {\omega_{d-1}}\int_0^{r-s/2}\cosh(t)\,\sinh^{d-2}(t)\,{\rm d}t,
	\end{align*}
which implies the asserted lower bound for $d=2$, since $\cosh(t)$ is the derivative of $\sinh(t)$ and $\sinh(t)\ge e^{t-3}$ for $t\ge 2$. For $d \geq 3$ we use now additionally that $\cosh(t)\geq e^t/2$ for all $t\geq 0$ and that $\sinh(t)\geq t$ for all $0\leq t\leq 2$. This gives
	\begin{align*}	
		\mathcal{H}^d(B(z,r) \setminus B(x,r)) &\geq s\omega_{d-1}\Big(\int_2^{r-s/2}\cosh(t)\,\sinh^{d-2}(t)\,{\rm d}t+\int_0^2\cosh(t)\,\sinh^{d-2}(t)\,{\rm d}t\Big)\\
		&\geq \frac{s\omega_{d-1}}{2}\Big( e^{-3(d-2)}\int_2^{r-s/2}e^{(d-1)t}\,{\rm d}t+\int_0^2 t^{d-2}\,{\rm d}t\Big)\\
		&= \frac{s\omega_{d-1}}{2(d-1)}\Big(\frac{e^{(d-1)(r-s/2)}}{e^{3(d-2)}}-e^{-d+4}+2^{d-1}\Big)\\
		&\geq \alpha_1se^{(d-1)(r-s/2)},
	\end{align*}
	where we used the fact that $2^{d-1}-e^{-d+4}>0$ for $d\geq 3$.
	
	To obtain an upper bound, let $\tilde{B}_r(z,x)$ be the $(d-1)$-dimensional ball with radius $r$ that is contained in the hyperbolic hyperplane through $B_r(z,x)$. We use the polar integration formula \eqref{eq:PolarIntegration} in  the hyperbolic hyperplane containing $B_r(z,x)$ and obtain
	\begin{align*}
		\mathcal{H}^d(B(z,r) \setminus B(x,r))&\le s 	\int_{\tilde{B}_r(z,x)}\cosh(d_h(L(y),m))\,\calh^{d-1}(\mathrm{d}y)\\ 
		&= s{\omega_{d-1}}\int_0^{r}\cosh(t)\,\sinh^{d-2}(t)\,{\rm d}t\\
		&\leq {s\omega_{d-1}\over 2^{d-2}}\int_0^{r}e^{(d-1)t}\,{\rm d}t\\
		&= {s\omega_{d-1}\over (d-1)2^{d-2}}(e^{(d-1)r}-1)\\
		&\leq \alpha_2 se^{(d-1)r},
	\end{align*}
	where we additionally applied the inequalities $\cosh(t)\leq e^t$ and $\sinh(t)\leq e^t/2$ for all $t\geq 0$. We thus conclude the proof of part (i). The inequalities in (ii) follow from \eqref{eq:VolBall} and the fact that $\sinh(t)=\frac{e^{t}-e^{-t}}{2}$ for $t \in \R$.
\end{proof}

\subsection{Poisson approximation}

In this section we rephrase a special case of \cite[Theorem 4.1]{BSY21} which we will use to prove Theorem \ref{Th1}. We work with a Polish space $\mathbb{X}$ and denote by ${\bf N}_\mathbb{X}$ the space of locally finite, simple counting measures on $\mathbb{X}$. As usual, we identify each element in ${\bf N}_\mathbb{X}$ with its support. Let $f:\XX\times{\bf N}_\XX\to\R$ and $g:\XX\times{\bf N}_\XX\to\{0,1\}$ be a measurable functions and define for $\mu\in{\bf N}_\XX$ the random point process
\begin{equation}\label{eq:XiMu}
	\xi[\mu] := \sum_{x\in\mu}g(x,\mu)\delta_{f(x,\mu)},
\end{equation}
whose intensity measure is denoted by $\E\xi[\mu](\,\cdot\,)$. Let $\mathcal F$ denote the system of closed sets in $\mathbb{X}$ endowed with the Fell topology, see \cite[p.\ 256]{last2017lectures}. To each point $x\in\XX$ we associate in a measurable way a stopping set $\mathcal S(x,\,\cdot\,):\mathbf{N}_{\mathbb{X}} \to \mathcal F$ as well as a closed set $S_{x}\subset\XX$ with $x\in S_x$, where we recall that the stopping property of $S(x,\,\cdot\,)$ means that $\{\mu\in\mathbf{N}_{\mathbb{X}}:S(x,\mu)\subseteq K\}=\{\mu\in\mathbf{N}_{\mathbb{X}}:S(x,\mu\cap K)\subseteq K\}$ for all compact $K\subseteq\mathbb{X}$. It is assumed that $f$ and $g$ are localized in the sense that for $\mathcal{S}(x,\mu) \subset S_x$,
\begin{align*}
	g(x,\mu) &= g(x,\mu\cap S_x)\\
	f(x,\mu) &= f(x,\mu\cap S_x)\quad\text{if}\quad g(x,\mu)=1.
\end{align*}

To rephrase the result from \cite{BSY21} we need the total variation distance $\mathbf{d_{TV}}(\nu_1,\nu_2)$ between two measures $\nu_1,\nu_2$ on $\XX$, which is defined as
$$
\mathbf{d_{TV}}(\nu_1,\nu_2) := \sup_{B}|\nu_1(B)-\nu_2(B)|,
$$
where the supremum runs over all Borel subsets $B$ of $\XX$ satisfying $\nu_1(B),\nu_2(B)<\infty$. Now, let $\eta$ be a Poisson process on $\XX$ with intensity measure $\E\eta(\,\cdot\,)$, and define the quantities
\begin{align*}
	E_1&:=\int_{\XX} \E[g(x,\eta+\delta_x) \mathds 1 \{\mathcal S(x,\eta)\not \subset S_x\}]	\,\E\eta({\rm d}x),\\
	E_2 &:= \int_{\XX}\int_{\XX}\mathds{1}\{S_x\cap S_z\neq\varnothing\}\E[g(x,\eta+\delta_x)]\E[g(z,\eta+\delta_z)]\,\E\eta({\rm d}z)\E\eta({\rm d}x),\\
	E_3 &:= \int_{\XX}\int_{\XX}\mathds{1}\{S_x\cap S_z\neq\varnothing\}\E[g(x,\eta+\delta_x+\delta_z)g(z,\eta+\delta_x+\delta_z)]\,\E\eta({\rm d}z)\E\eta({\rm d}x),
\end{align*}
where $\delta_{(\cdot)}$ denotes the Dirac measure. Assuming finally that $\E\xi[\eta](\R)<\infty$, \cite[Theorem 4.1]{BSY21} says that the Kantorovich-Rubinstein distance $\mathbf{d_{KR}}(\xi[\eta],\zeta)$ between $\xi[\eta]$ and a Poisson process $\zeta$ on $\R$ with finite intensity measure $\E\eta(\,\cdot\,)$ can be estimated from above by
\begin{align}\label{eq:BSYbound}
	\mathbf{d_{KR}}(\xi[\eta],\zeta) \leq \mathbf{d_{TV}}(\E\xi[\eta],\E\zeta) + 2(E_1+E_2+ E_3).
\end{align}
We refer to \cite{BSY21,DST16} for a a dual formulation as well as more details on the Kantorovich-Rubinstein distance between random points measures. Let us also remark that if in \eqref{eq:BSYbound} the Kantorovich-Rubinstein distance is replaced by the total variation distance, the result without the factor $2$ on the right-hand side can be found in  \cite{BarbourBrown}.

\section{Proof of Theorem \ref{Th1}}\label{sec:Proof2}

Our goal is to apply the Poisson approximation bound \eqref{eq:BSYbound}. To this end, we need to specify $\XX$, $\eta$, the functions $f$ and $g$ as well as the sets $S(x,\,\cdot\,)$ and $S_x$. For the space $\XX$ we take the $d$-dimensional hyperbolic space $\H^d$ and for $\eta$ a Poisson process on $\H^d$ with intensity measure $\calh^d$. To ensure finiteness of the involved intensity measures we fix some arbitrary $c\in \mathbb{R}$ and define the two functions $g,f:\H^d\times{\bf N}_{\H^d}\to\R$ by
\begin{align*}
	&g(x,\mu):=\mathds{1}\{x \in B_R\} \mathds{1}\{\mathcal{H}^d(B_{r(x,k,{\mu})})-v_k(R)>c\},\\
	&f(x,\mu):=\mathcal{H}^d(B_{r(x,k,{\mu})})-v_k(R),
\end{align*}
where we recall that $r(x,k,\mu)$ denotes the distance to the $k$th nearest neighbour of a point $x$ in the support of a simple counting measure $\mu$.
Then the point process of exceedances  \eqref{eq:PPExceedances} restricted to the interval $(c,\infty)$ has the same distribution as $\xi[\eta]$ in \eqref{eq:XiMu} using the functions $f$ and $g$ as just defined. Next, for $x \in \mathbb{H}^d$ let $\mathcal S(x,\mu):=B(x,r(x,k,\mu))$ and $S_x:=B(x,r_{c'})$ be the closed ball with centre at $x$ and $\mathcal{H}^d$-measure $c'+v_k(R)>0$ for some $c'> \max(c,0)$ to be specified below. We emphasize that this choice implicitly determines the radius $r_{c'}$ via \eqref{eq:VolBall}. By construction, the functions $f$ and $g$ are localized to the sets $\mathcal S(x,\mu)$. 

Taking $s:=r_{c'}$ in Lemma \ref{voldif}(ii) we obtain
		\begin{equation}\label{rcbou}
			\begin{split}
			\log (c'+v_k(R)) &\le r_{c'}(d-1) +\log\left(\frac{\omega_d}{(d-1)2^{d-1}}\right) \\
			&\le \log\left(\frac{(d-3)e^{2r_c'}}{(d-3)e^{2r_c'}-(d-1)^2}\right)+\log (c'+v_k(R)).
						\end{split}
		\end{equation}
		Finally, we let $\zeta$ be an inhomogeneous Poisson process on $\R$ with intensity measure $\E\zeta((u,\infty))=e^{-u}$, $u>c$, as in the statement of Theorem \ref{Th1}. We can now apply \eqref{eq:BSYbound} to conclude that for all $R>0$,
	\begin{align}\label{eq:5922A}
		\mathbf{d_{KR}}(\xi_R^{(k)} \cap (c,\infty),\zeta \cap (c,\infty)) \le \mathbf{d_{TV}}(\mathbb{E}\xi_R^{(k)} \cap (c,\infty), \mathbb{E}\zeta \cap (c,\infty))+2(E_1+E_2+E_3),
	\end{align}
	where the error terms $E_1$, $E_2$ and $E_3$ are given by
	\begin{align*}
		E_1&:=\int_{\H^d} \mathds 1\{x \in B_R\} \P[\mathcal{H}^d(B_{r(x,k,\eta)})>c+v_k(R),\,r(x,k,\eta) >r_{c'}]\,\mathcal{H}^d(\mathrm{d}x),\\
		E_2&:=\int_{\H^d}\int_{\H^d} \mathds{1}\{x,z \in B_{R}, B(x,r_{c'}) \cap B(z,r_{c'}) \neq \emptyset\}\\
		&\qquad\qquad \times \P[\mathcal{H}^d(B_{r(x,k,\eta)})>c+v_k(R)] \P[\mathcal{H}^d(B_{r(z,k,\eta)})>c+v_k(R)]\, \mathcal{H}^d(\mathrm{d}z) \mathcal{H}^d(\mathrm{d}x),\\
		E_3&:=\int_{\H^d}\int_{\H^d} \mathds{1}\{x,z \in B_{R}, B(x,r_{c'}) \cap B(z,r_{c'}) \neq \emptyset\}\\
		&\qquad\qquad \times \P[\mathcal{H}^d(B_{r(x,k,\eta+\delta_z)})>c+v_k(R),\,\mathcal{H}^d(B_{r(z,k,\eta+\delta_x)})>c+v_k(R)]\, \mathcal{H}^d(\mathrm{d}z) \mathcal{H}^d(\mathrm{d}x).
	\end{align*}
	The remaining parts of the proof bound individually the four terms on the right hand side of \eqref{eq:5922A}.
	
	\paragraph{Bounding the total variation distance.}
	In a first step we investigate the intensity measure of $\xi_R^{(k)}$. Let $r_c>0$ be such that $\mathcal{H}^d(B_{r_c})=v_k(R)+c$ and note that $\mathcal{H}^d(B_{r(x,k,\mu)})>v_k(R)+c$ if and only if $\mu(B(x,r_c))\le {k-1}$. From the Mecke equation for Poisson processes \cite[Theorem 4.1]{last2017lectures} we obtain that for all $u>c$,
	\begin{align*}
		\mathbb{E}\xi_R^{(k)}((u,\infty))&=\int_{\H^d} \mathds{1}\{x \in B_{R}\} \P[\mathcal{H}^d(B_{r(x,k,\eta)})>u+v_k(R)]\, \mathcal{H}^d(\mathrm{d}x).
	\end{align*}
	Since the distribution of $\eta$ is invariant under hyperbolic isometries and $\mathcal{H}^d(B_{r(x,k,\eta)})>u+v_k(R)$ if and only if there are at most ${k-1}$ points of $\eta$ in a ball with $\mathcal{H}^d$-measure $u+v_k(R)$ around $x$, the expression is equal to
	\begin{align*}
		&\mathcal{H}^d(B_R) e^{-u-v_k(R)}\sum_{\ell=0}^{{k-1}} \frac{(u+v_k(R))^\ell}{\ell!}\\
		&\,=e^{-u}\mathcal{H}^d(B_R)  \frac{(k-1)!2^{d-1} (d-1)e^{-R(d-1)}}{\omega_d(R(d-1))^{k-1}} \sum_{\ell=0}^{{k-1}} \frac{\big(u+v_k(R)\big)^\ell}{\ell!},
	\end{align*}
	where we used the definition \eqref{eq:vkr} of $v_k(R)$ and the Poisson property of $\eta$. Hence, the Lebesgue density $\varrho$ of $\mathbb{E} \xi_R^{(k)}\cap (c,\infty)$ is given by
	\begin{align*}
		\varrho(u)=e^{-u}\mathcal{H}^d(B_R) \frac{2^{d-1} (d-1)e^{-R(d-1)}(u+v_k(R))^{k-1}}{\omega_d(R(d-1))^{k-1}},\qquad u>c.
	\end{align*}
	Now, \eqref{volbou} gives for $R>0$ large enough
	\begin{align}
		\textbf{d}_{\textbf{TV}}(\mathbb{E}\xi_R^{(k)} \cap (c,\infty),\mathbb{E}\zeta \cap (c,\infty)) &\le \int_c^\infty |\varrho(u)-e^{-u}|\,\mathrm{d}u\nonumber\\
		&=  \int_c^\infty \Big| e^{-u}\mathcal{H}^d(B_R) \frac{2^{d-1} (d-1)e^{-R(d-1)}(u+v_k(R))^{k-1}}{\omega_d(R(d-1))^{k-1}}-e^{-u}\Big|\, \mathrm{d}u \nonumber\\
		&\le e^{-c} \Big[(1+\beta_1e^{-2R})\Big(1+\frac{c+\beta_2\log R}{R(d-1)}\Big)^{k-1}-1\Big]\nonumber\\
		&\le\begin{cases}
			\beta_1 e^{-c} e^{-2R} &: k=1\\
			\beta_3 e^{-c}(c+\beta_2\log R) R^{-1} &: k\geq 2
		\end{cases}\label{tvbou}
	\end{align}
	for constants $\beta_1,\beta_2, \beta_3>0$ only depending on $k$ and $d$.	

		\paragraph{Bounding $\mathbf E_1$.}
	Note that $r(x,k,\eta) >r_{c'}$ if and only if $\mathcal{H}^d(B_{r(x,k,\eta)})>c'+v_k(R)$. Hence, we find from \eqref{volbou} similarly to the estimate leading to \eqref{tvbou} that for all $c'>\max(c,0)$,
		\begin{align}
			E_1=\mathbb{E}\xi_R^{(k)} \cap (c',\infty)\le \begin{cases}
				(1+\beta_1e^{-2R})e^{-c'}&: k=1\\
				\beta_4 e^{-c'} 
		\Big(\frac{c'+v_k(R)}{R(d-1)}\Big)^{k-1}&: k\geq 2
		\end{cases}\label{E1bou}
		\end{align}
	where the constant $\beta_4>0$ depends only on $k$.
	\paragraph{Bounding $\mathbf E_2$.}
	Since $B(x,r_c) \cap B(z,r_c) \neq \emptyset$ if and only if the hyperbolic distance of $x$ and $z$ is at most $2r_c$, we obtain from the invariance of $\eta$ under hyperbolic isometries that $E_2$ is bounded by
	\begin{align*}
		&\mathbb{E} \xi_R^{(k)}((c,\infty)) \mathbb{P}[\mathcal{H}^d(B_{r(p,k,\eta)})>c+v_k(R)] \int_{\H^d} \mathds{1}\{B(p,r_{c'}) \cap B(z,r_{
				c'}) \neq \emptyset\}\, \mathcal{H}^d({\rm d}z)\\
		&\quad \leq (\mathbb{E} \xi_R^{(k)}((c,\infty)))^2 \frac{\mathcal{H}^d(B_{2r_{c'}})}{\mathcal{H}^d(B_{R})}.
	\end{align*}
		From \eqref{rcbou} we have for $R>0$ large enough that
		\begin{equation}\label{eq:EstimateRc}
			r_{c'}\le \frac{1}{d-1} \log (c'+v_k(R))+\beta_5
		\end{equation}
		for some constant $\beta_5>0$ only depending on $d$. Hence, using the definition of $v_k(R)$, there is another constant $\beta_6>0$ only depending on $d$ such that for all $R>0$,
		\begin{align}
			E_2 \le \beta_6 (c'+R)^2 e^{-R(d-1)}. \label{E2bou}
		\end{align}
	
	\paragraph{Bounding $\mathbf E_3$.}
		First we consider the case $k=1$. Note that $(\eta-\delta_x)(B(x,r_c))=0$ and $(\eta-\delta_z)(B(z,r_c))=0$ implies  for $x,z \in \eta$ that $d_h(z,x)\ge r_c$. Hence, we obtain for $E_3$ the upper bound
	\begin{align}
		E_3\leq &\int_{\H^d}\int_{\H^d} \mathds{1}\{z \in B_{R}, x \in B_R \cap B(z,2r_{c'})\setminus B(z,r_c)\} \P[\eta(B(z,r_c))=0]\nonumber\\
		&\qquad \qquad  \times \P[\eta(B(z,r_c)\setminus B(x,r_c))=0]\,\mathcal{H}^d(\mathrm{d}z) \mathcal{H}^d(\mathrm{d}x)\nonumber\\
		& \le \mathbb{E} \xi_R^{(1)}((c,\infty)) \int_{\mathbb{H}^d} \mathds{1}\{x \in B(p,2r_{c'})\setminus B(p,r_c)\}  \P[\eta(B(z,r_c)\setminus B(x,r_c))=0]\, \mathcal{H}^d(\mathrm{d}x).\label{E2k1}
	\end{align}
	Let $z \in \H^d\setminus B(p,r_c)$ and observe that $\mathcal{H}^d(B(z,r_c)\setminus B(p,r_c)) \ge \mathcal{H}^d(B_{r_c})/2$. Thus, 
	\begin{align*}
		\P[\eta(B(z,r_c)\setminus B(x,r_c))=0] = \exp(-\mathcal{H}^d(B(z,r_c)\setminus B(x,r_c))) \le  \exp(-\mathcal{H}^d(B_{r_c})/2).
	\end{align*} 
	This gives for \eqref{E2k1} the bound
	\begin{align} \label{eg:E2bouk1}
		\mathbb{E} \xi_R^{(1)}((c,\infty)) \mathcal{H}^d(B_{2r_{c'}})e^{-R(d-1)/2}\leq	\beta_{7} e^{2(d-1)r_{c'}}e^{-R(d-1)/2}
	\end{align}
	for some constant $\beta_7>0$ only depending on $d$.
	
	Next we consider the case $k \ge 2$ and let $a \in (0,1]$, its precise value will be specified later. In order to bound $E_3$ we distinguish {the situations that $z \in B(x,ar_{c'})$ and $z \notin B(x,ar_{c'})$. In the first case we have that $\mathcal{H}^d(B_{r(x,k,\eta+\delta_z)})>v_k(R)+c$ and $\mathcal{H}^d(B_{r(z,k,\eta+\delta_x)})>v_k(R)+c$ if and only if $\eta(B(x,r_c))\le {k-2}$ and $\eta(B(z,r_c))\le {k-2}$. This allows us to bound $E_3$ by}
	\begin{align}
		&\int_{\H^d}\int_{\H^d} \mathds{1}\{x \in B_{R}, z \in B_R \cap B(x,{ar_{c'}})\} \P[\eta(B(x,r_c))\le {k-2}]\nonumber\\
		&\qquad \qquad  \times \P[\eta(B(z,r_c)\setminus B(x,r_c))\le {k-2}]\,\mathcal{H}^d({\rm d}z) \mathcal{H}^d({\rm d}x)\label{E2a}\\
		& + \int_{\H^d}\int_{\H^d} \mathds{1}\{x \in B_{R}, z \in B_R \cap (B(x,2r_{c'})\setminus B(x,{ar_{c'}}))\} \P[\eta(B(x,r_c))\le {k-1}]\nonumber\\
		&\qquad \qquad  \times \P[\eta(B(z,r_c)\setminus B(x,r_c))\le {k-1}]\,\mathcal{H}^d({\rm d}z) \mathcal{H}^d({\rm d}x)\label{E2b}.
	\end{align}
	By invariance of $\eta$ under hyperbolic isometries, the first term \eqref{E2a} is equal to 
	\begin{align}
		&\mathbb{E} \xi_{R}^{(k)}((c,\infty)) \frac{\P[\eta(B_{r_c})\le {k-2}]}{\P[\eta(B_{r_c})\le {k-1}]} \int_{\H^d} \mathds{1}\{z \in B_{ar_{c'}}\} \P[\eta(B(z,r_c)\setminus B_{r_c})\le {k-2}]\, \mathcal{H}^d(\mathrm{d}z)\nonumber\\
		&\quad= \omega_d \mathbb{E} \xi_{R}^{(k)}((c,\infty))\frac{\P[\eta(B_{r_c})\le {k-2}]}{\P[\eta(B_{r_c})\le {k-1}]}\nonumber\\
		&\quad \quad \times \sum_{\ell=0}^{{k-2}}  \int_0^{{ar_{{c'}}}} \sinh^{d-1}(s) \exp\left(-\calh^d(B(z,r_c)\setminus B_{r_c})\right) \frac{(\calh^d(B(z,r_c)\setminus B_{r_c}))^\ell}{\ell!}\, \mathrm{d}s,\label{E2aint}
	\end{align}
	where we applied the polar integration formula \eqref{eq:PolarIntegration} with $z:=\exp_p(sv)$ for some arbitrary $v \in \mathbb{S}_p^{d-1}$ in \eqref{E2aint}.	Note that we always have that $2(r_c-2)>r_c\geq s$ if $r_c\geq 5$, which can be achieved by choosing the parameter $R$ large enough. Thus, for such $R$ we find by Lemma \ref{lem:LowerBoundVolumeBall} and Lemma \ref{voldif}(i) that
	\begin{align}
		&\int_0^{{ar_{{c'}}}} \sinh^{d-1}(s) \exp\left(-\calh^d(B(z,r_c)\setminus B_{r_c})\right) \frac{(\calh^d(B(z,r_c)\setminus B_{r_c}))^\ell}{\ell!}\, \mathrm{d}s\nonumber\\
		&\leq {1\over 2^{d-1}}\int_0^{{ar_{c'}}} e^{(d-1)s}\,\exp\Big(-{\alpha_1se^{(d-1)(r_c-{s\over 2})}}\Big) {(\alpha_2 s e^{(d-1)r_c})^\ell\over\ell!}\,\mathrm{d}s\nonumber\\
		&\leq {1\over 2^{d-1}}{(\alpha_2 e^{(d-1)r_c})^\ell\over\ell!}\int_0^{\infty} \exp\Big(-s\Big(\alpha_1e^{{(d-1)(2r_c-ar_{c'}) \over 2}}-(d-1)\Big)\Big) s^\ell\,\mathrm{d}s\nonumber\\
		&\leq \beta_8 {\exp\Big( - (d-1)\frac{2r_c-a(\ell+1)r_{c'}}{2} \Big)},\label{eq:csc}
	\end{align}
	which takes the maximum value $\beta_{8} e^{- (d-1)\frac{2r_c-a(k-1)r_{c'}}{2}}$ for {$\ell=k-2$}, where $\beta_{8}>0$ is a constant only depending on $d$.

	Next, we recall that for a general Poisson random variable $X$ with mean $\lambda>0$ and any bounded function $h:\{0,1,\ldots\}\to\R$ one has that
	\begin{align*}
	\E[\lambda h(X+1)]=\E[Xh(X)];
	\end{align*}
	in fact, this is the famous Chen-Stein characterization of the Poisson distribution. In particular, applying this to the function $h(n)={\bf 1}\{n\leq k\}$ we see that
	\begin{align*}
		\lambda\P(X\leq k-1) = \E[X{\bf 1}\{X\leq k\}]\leq k\P(X\leq k).
	\end{align*}
	Applying this to the Poisson random variable $\eta(B_{r_c})$ which has mean $\calh^d(B_{r_c})=v_k(R)+c$ we find from \eqref{eq:vkr} that 
	$$
	\frac{\P[\eta(B_{r_c})\le {k-2}]}{\P[\eta(B_{r_c})\le {k-1}]} \le {k-1\over v_k(R)+c}.
	$$
	
	 {Hence, we obtain from \eqref{E2aint} and \eqref{eq:csc}  that \eqref{E2a} is bounded by
		\begin{align*}
			\beta_9 \mathbb{E} \xi_{R}^{(k)}((c,\infty)){k-1\over v_k(R)+c} e^{- (d-1)\frac{2r_c-a(k-1)r_{c'}}{2}} \leq \beta_{10} R^{-1} e^{- (d-1)\frac{2r_c-a(k-1)r_{c'}}{2}},
		\end{align*}
		where $\beta_9, \beta_{10}>0$ are constants only depending on $d$ and $k$.
	}
	
	Now, we consider \eqref{E2b}. By invariance of $\eta$ under hyperbolic isometries, \eqref{E2b} is the same as
	\begin{align*}
		&\mathbb{E} \xi_{R}^{(k)}((c,\infty))  \int_{\H^d} \mathds{1}\{z \in B_{2r_{c'}} \setminus B_{ar_{c'}}\} \P[\eta(B(z,r_c)\setminus B_{r_c})\le {k-1}]\, \mathcal{H}^d(\mathrm{d}z)\nonumber\\
		&= \omega_d \mathbb{E} \xi_{R}^{(k)}((c,\infty)) \sum_{\ell=0}^{{k-1}}  \int_{ar_{c'}}^{2r_{c'}} \sinh^{d-1}(s) \exp\left(-\calh^d(B(z,r_c)\setminus B_{r_c})\right) \frac{(\calh^d(B(z,r_c)\setminus B_{r_c}))^\ell}{\ell!}\, \mathrm{d}s,
	\end{align*}
	where we again applied the polar integration formula \eqref{eq:PolarIntegration} with $z:=\exp_p(sv)$ for some arbitrary $v \in \mathbb{S}_p^{d-1}$.
	Using that $\calh^d(B(z,r_{c'})\setminus B_{r_c}) \ge \calh^d(B_{s/2})$ for $d_h(p,z)=s$, we find from Lemma \ref{lem:LowerBoundVolumeBall} that $\calh^d(B(z,r_c)\setminus B_{r_c}) \ge \gamma_{d}e^{(d-1)s/2}$ (the result can indeed be applied, since {$s/2 \geq ar_c/2 \geq 2$} for large enough $R$), see Figure \ref{fig:k=1}. Since the function $u \mapsto e^{-u}u^\ell$ is decreasing for $u \ge u_0$ large enough, we obtain for {$M>0$ to be specified below and} $R$ large enough, 
		\begin{align}
			\nonumber &\int_{ar_{c'}}^{2r_{c'}} \sinh^{d-1}(s) \exp\left(-\calh^d(B(z,r_c)\setminus B_{r_c})\right) \frac{(\calh^d(B(z,r_c)\setminus B_{r_c}))^\ell}{\ell!}\, \mathrm{d}s\\
			\nonumber &\quad \le \frac{1}{2^{d-1}} \int_{ar_{c'}}^{2r_{c'}} e^{(d-1)s} \exp\Big(-\gamma_dse^{(d-1)s/2} \Big) \frac{(\gamma_dse^{(d-1)s/2})^\ell}{\ell!} \mathrm{d}s\\
			\nonumber &\quad \le \frac{ e^{-r_{c'}\frac{M(d-1)a}{2}}}{2^{d-2}(d-1)(ar_{c'})^{M+2}} \int_{ar_{c'}}^{2r_{c'}} \Big(1+{s(d-1) \over 2}\Big)e^{(d-1)s/2} \exp\Big(-\gamma_dse^{(d-1)s/2} \Big) \frac{\gamma_d^\ell(se^{(d-1)s/2})^{\ell+1+M}}{\ell!}\mathrm{d}s\\
			\nonumber &\quad \le \beta_{11} r_{c'}^{-M-2} e^{-r_{c'}\frac{M(d-1)a}{2}},
		\end{align}
	where we have used the substitution $t=se^{(d-1)s/2}$ for the last inequality and $\beta_{11}>0$ is a constant depending on $d$ and on $M$ .
	\begin{figure}[t]
		\centering
		\vspace{-2cm}
		\includegraphics[width=\columnwidth]{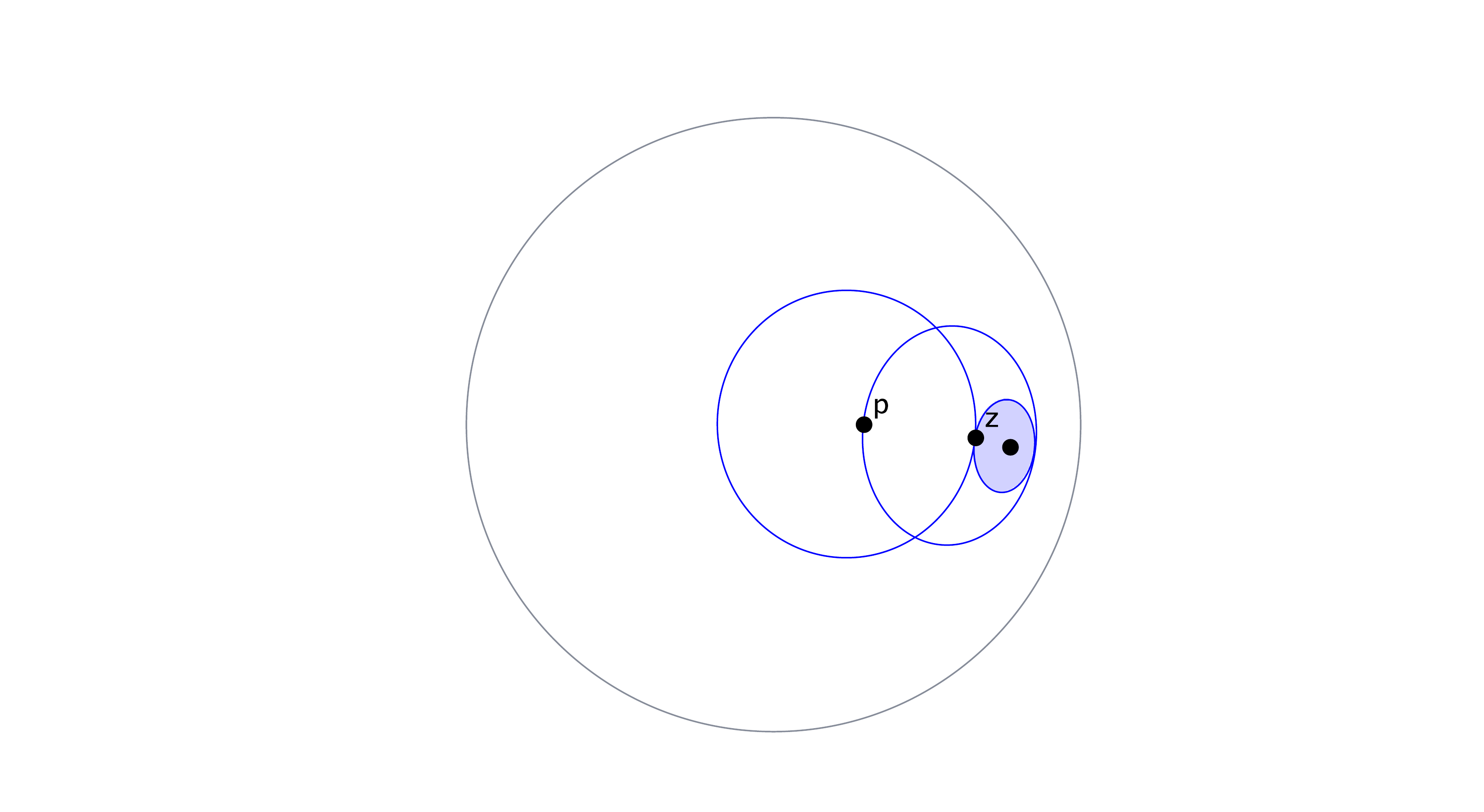}
		\caption{Illustration in the Beltrami-Klein model for the hyperbolic plane of the argument used in the estimate of \eqref{E2b}. Shown are the balls $B(p,r_c)$ and $B(z,r_c)$, where $d_h(p,z)=r_c$. The blue ball has radius $r_c/2$.}
		\label{fig:k=1}
	\end{figure}
	Summarizing, we have shown that for large enough $R$,
	\begin{align}\label{eq:6323a}
		E_3 \leq \begin{cases}
			\beta_{7} e^{2(d-1)r_{c'}}\exp(-\alpha_1 r_c e^{(d-1)r_c/2}) &: k=1\\
			\beta_{10} R^{-1}{\exp\left( - (d-2)\frac{2r_c-a(k-1)r_{c'}}{2} \right)} +{\beta_{11} e^{-r_{c'}\frac{M(d-1)a}{2}}} &: k\geq 2.
		\end{cases}
	\end{align}
	{Thus, choosing $a :={r_c \over{k r_{c'}}}\in (0,1]$} and $M>0$ so large that $Ma>2$, we find for $k \ge 2$ that 
		\begin{align}
			E_3 \le \beta_{12}R^{-1},\label{eq:E2bou}
		\end{align}
		where $\beta_{12}$ is a constant depending on $c$, $d$ and on $k$.
	
	\paragraph{Completing the proof.} After having estimated all terms in \eqref{eq:5922A} we first conclude for $k=1$ from \eqref{tvbou}, \eqref{E1bou}, \eqref{E2bou} and \eqref{eq:6323a} that for all sufficiently large $R$,
\begin{align*}
	&\mathbf{d_{KR}}(\xi_R^{(1)} \cap (c,\infty),\zeta \cap (c,\infty))\\
	&\quad\le \beta_1e^{-c} e^{-2R}+2\Big\{	(1+\beta_1e^{-2R})e^{-c'}
	+\beta_9 (c'+R)^2 e^{-R(d-1)} +	\beta_{7} e^{2(d-1)r_{c'}}e^{-R(d-1)/2}\Big\}.
\end{align*}
Choosing $c'=c+\log R$ and using \eqref{eq:EstimateRc} we see that
\begin{align*}
\mathbf{d_{KR}}(\xi_R^{(1)} \cap (c,\infty),\zeta \cap (c,\infty)) \le \begin{cases}
	C_{1,d} \,Re^{-R(d-1)/2} &: d\leq 5\\
	C_{1,d} \,e^{-2R} &: d\geq 6.
\end{cases}
\end{align*}
where $C_{1,d}>0$ is some constant only depending on $c$ and on $d$, which is the result of Theorem \ref{Th1}(i).

On the other hand, for $k\geq 2$ and using again \eqref{tvbou}, \eqref{E1bou}, \eqref{E2bou} and this time \eqref{eq:E2bou}  we arrive at the bound
\begin{align*}
	&\mathbf{d_{KR}}(\xi_R^{(k)} \cap (c,\infty),\zeta \cap (c,\infty))\\
	&\quad\le	\beta_3 e^{-c}(c+\beta_1\log R) R^{-1} +2\Big\{\beta_4 e^{-c'} 
	\Big(\frac{c'+v_k(R)}{R(d-1)}\Big)^{k-1}+\beta_9 (c'+R)^2 e^{-R(d-1)} +\beta_{12}R^{-1}\Big\}.
\end{align*}
Choosing again $c'=c+\log R$ we conclude that
\begin{align*}
\mathbf{d_{KR}}(\xi_R^{(k)} \cap (c,\infty),\zeta \cap (c,\infty)) \le C_{k,d} \,{\log R \over R},
\end{align*}
where $C_{k,d}>0$ is some constant depending on $c$, $k$ and on $d$. This is the assertion of Theorem \ref{Th1}(ii) and completes the proof.
\qed

\section*{Acknowledgements}

The authors wish to thank two anonymous referees for their helpful suggestions which led to an improvement of this article. 
This work has been supported by the DFG priority program SPP 2265 \textit{Random Geometric Systems}.

 \end{document}